\documentclass[11pt,reqno]{amsart}
\usepackage{verbatim,paralist}
\usepackage[breaklinks,pdfstartview=FitH]{hyperref}

\usepackage{amssymb,mathrsfs,verbatim,graphicx,paralist}

\renewcommand{\le}{\leqslant}

\renewcommand{\ge}{\geqslant}
\renewcommand{\leq}{\leqslant}
\renewcommand{\geq}{\geqslant}

\newcommand{\e}{\varepsilon}

\theoremstyle{plain}
  \newtheorem{lemma}{Lemma}
  \newtheorem{theorem}[lemma]{Theorem}
  \newtheorem{corollary}[lemma]{Corollary}

  \theoremstyle{definition}

  \newtheorem{observation}[lemma]{Observation}
  \newtheorem{remark}[lemma]{Remark}

\newcommand{\pr}{\mathbb{P}}

\newcommand{\jnote}[1]{}

\newcommand{\E}{{\mathbb E}}

\newcommand{\Lip}{\mathrm{Lip}}

\newcommand{\dist}{\mathsf{dist}}

\newcommand{\supp}{\mathrm{supp}}
\newcommand{\remove}[1]{}

\newcommand{\eps}{\varepsilon}

\newcommand{\ceiling }[1]{{\left\lceil{#1}\right\rceil}}

\newcommand{\I}{{\textrm{I}}}
\newcommand{\II}{{\textrm{II}}}
\newcommand{\III}{{\textrm{III}}}
\newcommand{\IV}{{\textrm{IV}}}

\begin{document}

\title{A lower bound on dimension reduction for trees in $\ell_1$}
\author{James R. Lee}
\address{Computer Science \& Engineering, University of Washington} \email{jrl@cs.washington.edu}
\author{Mohammad Moharrami}
\address{Computer Science \& Engineering, University of Washington} \email{mohammad@cs.washington.edu}

\begin{abstract} There is a constant $c > 0$ such that
for every $\eps \in (0,1)$ and $n \geq 1/\e^2$, the following holds.
Any mapping from the $n$-point star metric into $\ell_1^d$
with bi-Lipschitz distortion $1+\eps$ requires dimension $$d \geq {c\log n\over \eps^2\log (1/\eps)}\,.$$
\end{abstract}

\maketitle


\section{Introduction}

Consider an integer $n \geq 1$.  The {\em $n$-node star} is the simple, undirected graph $G_n = (V_n, E_n)$
with $|V_{n}|=n$, where one node has degree $n-1$ and all other nodes have degree one.  We write $\rho_n$
for the shortest-path metric on $G_n$ where each edge is
equipped with a unit weight.
We use $\ell_1^d$ to denote the space $\mathbb R^d$ equipped with the $\ell_1$ norm.
Our main theorem follows.

\begin{theorem}\label{thm:main}
There is a constant $c > 0$ such that the following holds.
Consider any $\eps\in(0,\frac{1}{16})$ and  $n\geq  {1/\eps^2}$.
Suppose there exists a $1$-Lipschitz mapping $f : V_n \to \ell_1^d$
such that $\|f(x)-f(y)\|_1 \geq (1-\e) \rho_n(x,y)$ for all $x,y \in V_n$.
Then,
$$
d \geq  {c\log n\over{\eps^2\log(1/\eps)}}\,.
$$
\end{theorem}

One can achieve such a mapping with $d \leq O\left(\frac{\log n}{\e^2}\right)$, thus the
theorem is tight up to the factor of $c/\log(1/\eps)$.  In general, de Mesmay and the authors \cite{JMMfull}
proved that every $n$-point tree metric admits a distortion $1+\varepsilon$ embedding into $\ell_1^{C(\e) \log n}$
where $C(\e) \leq O((\frac{1}{\eps})^4 \log \frac{1}{\eps})$.  For the special case
of complete trees where all internal nodes have the same degree (such as the $n$-star), they achieve $C(\eps) \leq O(\frac{1}{\eps^2})$.

We recall that given two metric spaces $(X,d_X)$ and $(Y,d_Y)$ and a map $f : X \to Y$, one defines the {\em Lipschitz constant of $f$} by
by $$\|f\|_{\Lip} = \sup_{x \neq y \in X} \frac{d_Y(f(x),f(y))}{d_X(x,y)}\,.$$
The {\em bi-Lipschitz distortion of $f$} is the quantity $\dist(f) = \|f\|_{\Lip} \cdot \|f^{-1}\|_{\Lip}$,
which is taken as infinite when $f$ is not one-to-one.  If there exists such a map
$f$ with distortion $D$, we say that {\em $X$ $D$-embeds into $Y$.}

A {\em finite tree metric} is a finite, graph-theoretic
tree $T=(V,E)$, where every edge is equipped with a positive length.  The metric $d_T$ on $V$ is
given by taking shortest paths.  Since every finite tree metric embeds isometrically into $\ell_1$,
one can view the preceding statements as quantitative bounds on the dimension
required to achieve such an embedding with small distortion (instead of isometrically).

\medskip

Such questions have a rich history.  Perhaps most famously, if $X$ is an $n$-point subset of $\ell_2$,
then a result of Johnson and Lindenstrauss \cite{JL84} states that $X$ admits
a $(1+\e)$-embedding into $\ell_2^d$ where $d = O\left(\frac{\log n}{\e^2}\right)$.
Alon \cite{Alon03} proved that this is tight up to a $\log(1/\e)$ factor:  If $X \subseteq \ell_2^n$ is an orthonormal basis,
then any
$D$-embedding of $X$ into $\ell_1^d$ requires $d \geq \frac{\Omega(\log n)}{\e^2 \log(1/\e)}$.

The situation for finite subsets of $\ell_1$ is quite a bit more delicate.
Talagrand \cite{Tal90}, following earlier results of Bourgain-Lindenstrauss-Milman \cite{BLM89}
and Schechtman \cite{Sch87}, showed that every $n$-dimensional subspace $X \subseteq \ell_1$
(and, in particular, every $n$-point subset)
admits a $(1+\e)$-embedding into $\ell_1^d$, with $d \leq O(\frac{n \log n}{\e^2})$.
For $n$-point subsets, this was improved to $d \leq O(n/\e^2)$ by Newman and Rabinovich \cite{NR10},
using the spectral sparsification techniques of Batson, Spielman, and Srivastava \cite{BSS09}.

In contrast, Brinkman and Charikar \cite{BC05} proved that there exist
$n$-point subsets $X \subseteq \ell_1$ such that any $D$-embedding of $X$ into $\ell_1^d$
requires $d \geq n^{\Omega(1/D^2)}$ (see also \cite{LN-diamond} for a simpler argument).
Thus the exponential
dimension reduction achievable in the $\ell_2$ case cannot be matched
for the $\ell_1$ norm.
More recently, it has been show by Andoni, Charikar, Neiman, and Nguyen \cite{ACNN11} that
there exist $n$-point subsets such that any $(1+\e)$-embedding requires dimension
at least $n^{1-O(1/\log(\e^{-1}))}$.
Regev \cite{Regev11} has given an elegant proof of both
these lower bounds based on information theoretic arguments.
Our proof takes some inspiration from Regev's approach.

\medskip

We note that Theorem \ref{thm:main} has an analog in coding theory.
Let $U_n = \{e_1, e_2, \ldots, e_n\} \subseteq \ell_1$.
Then any $(1+\e)$-embedding of $U_n$ into the Hamming cube $\{0,1\}^d$
requires $d \geq \frac{\Omega(\log n)}{\e^2 \log(1/\e)}$.  This
was proved in 1977 by McEliece, Rodemich, Rumsey, and Welch \cite{MRRW77}
using the Delsarte linear programming bound \cite{Delsarte73}.
The corresponding coding question concerns the maximum number
of points $x_1, x_2, \ldots \in \{0,1\}^d$ which satisfy
$(1-\e) d/2 \leq \|x_i-x_j\|_1 \leq (1+\e) d/2$ for $i \neq j$.
Alon's result for $\ell_2$ \cite{Alon03} yields
this bound as a special case since $\|x-y\|_2 = \sqrt{\|x-y\|_1}$
when $x,y \in \{0,1\}^d$.

On the one hand, the lower bound of Theorem \ref{thm:main} is stronger
since it applies to the target space $\ell_1^d$ and not simply $\{0,1\}^d$.
On the other hand, it is somewhat weaker since embedding $U_n$
corresponds to embedding only the leaves of the star graph $G_n$,
while our lower bound requires an embedding of the internal vertex
as well.  In fact, this is used in a fundamental and crucial way in our proof.
Still, in Section \ref{sec:kahane},
we prove the following somewhat weaker lower bound using simply the set $U_n$.

\begin{theorem}\label{thm:kahane}
There is a constant $c > 0$ such that
for every $\e \in (0,1)$, for all $n$ sufficiently large, any $(1+\e)$-embedding
of $U_n \subseteq \ell_1^n$ into $\ell_1^d$ requires
$$
d \geq \frac{c \log n}{\e \log \frac{1}{\e}}\,.
$$
\end{theorem}

For the case of isometric embeddings (i.e., $\e=0$), Alon and Pudl\'ak \cite{AP03}
showed that if $U_n$ embeds isometrically in $\ell_1^d$, then $d \geq \Omega(n/(\log n))$.
Our proof of Theorem \ref{thm:kahane} bears some similarity to their approach.

\medskip

Finally, we mention that if $B_h$ denotes the height-$h$ complete binary tree
(which has $2^{h+1}-1$ nodes), then it was proved by
Charikar and Sahai \cite{CS02} that for every $h \geq 1$ and $\e > 0$, $B_h$ admits a $(1+\e)$-embedding
into $\ell_1^d$ with $d \leq O(h^2/\e^2)$.
It was asked in \cite{MatOpen} whether one could achieve $d \leq O(h/\e^2)$
and this was resolved positively in \cite{JMMfull}.
From Theorem \ref{thm:main}, one can deduce that this upper bound
is asymptotically tight up to the familiar factor
of $\log(1/\e)$.
This corollary is proved in Section \ref{sec:kary}.

\begin{corollary}\label{cor:kary}
For any $\e > 0$ and $k \geq 2$, the following holds.
For $h$ sufficiently large, any $(1+\e)$-embedding of the complete $k$-ary, height-$h$
tree into $\ell_1^d$ requires
$$
d \geq \frac{\Omega(h \log k)}{\e^2 \log(1/\e)}\,.
$$
\end{corollary}

\section{proof of Theorem~\ref{thm:main}}


We will first  bound the number of ``almost disjoint'' probability  measures  that can be put on a finite set. Then we will translate this to  a lower bound for the dimension required for embedding the $n$-star into $\ell_1^d$ with distortion $1+\eps$.

Let  $X$ be a finite ground set, and let $\mathcal S$ be a set of measures $X$. We say that $\mathcal{S}$ is \emph{$\eps$-unrelated} if, for all distinct elements $\mu, \nu\in \mathcal S$, $$\|\mu-\nu\|_{TV} \geq \frac 1 2(\mu(X)+\nu(X))-\eps,$$ where $\|\cdot\|_{TV}$ denotes the total variation distance.
The following lemma is an easy corollary of a fact from \cite{Regev11}. We include the proof here for completeness.

\begin{lemma}\label{lem:reduction}
For every $\eps\in(0,1)$ and $k\in \mathbb N$, if there exists a map $f : (V_n,\rho_n) \to \ell_1^k$ with distortion $1+\eps$, then there exists an $\eps$-unrelated set of probability measures on $\{1,\ldots ,2k+1\}$ of size $n-1$.
\end{lemma}

\begin{proof}
Let $r\in V_n$ denote the the vertex of degree $n-1$. By translation and scaling, we may assume that $f(r)=0$ and $f$ is $1$-Lipschiz.
Thus for all vertices $v\in V_n$, we have $\|f(v)\|_1\leq 1$. For each vertex $v\in V_n\setminus \{r\}$ define the measure $\mu_v$ as follows

$$
\mu_v(\{i\})=
\begin{cases}
\max(0,f(v)_i) & 1\leq i\leq k\\
\max(0,-f(v)_i) & k+1\leq i\leq 2k\\
1- \|f(v)\|_1& i=2k+1,
\end{cases}
$$
where we use $f(v)_i$ to denote the $i$th coordinate of $f(v)$.

Note that for all $u,v\in V_n\setminus \{r\}$ we have
\begin{align*}
\|\mu_u-\mu_v\|_{TV} &= \frac 12\left(\vphantom{\bigoplus}\|f(u)-f(v)\|_1+\left|\vphantom{\bigoplus}(1-\|f(u)\|_1)-(1-\|f(u)\|_1)\right|\right) \\
&\geq \|f(u)-f(v)\|_1\,.
\end{align*}
Since $f$ has distortion $1+\eps$, for any two distinct vertices $u,v\in V_n$, we have $$\|f(u)-f(v)\|_1\geq \left(2\over 1+\eps\right)\geq 2(1-\eps).$$
Therefore the collection $\{ \mu_v : v \in V_n \setminus \{r\} \}$ satisfies
the conditions of the lemma.
\end{proof}



The next lemma is the final ingredient that we need to prove Theorem~\ref{thm:main}. Let ${\mathcal M}_k$ be the set of all measures $\{1,2,\ldots, k\}$, and let $\mathcal P_k$ be the set of all probability measures on $\{1,2,\ldots, k\}$.

\begin{lemma}\label{lem:one}
There exists a universal constant $C \geq 1$ such that for $\eps\leq 1/16$, the following holds. If there is an $\eps$-unrelated set $\mathcal S\subseteq \mathcal P_k$, then there exists a $\frac 1 2$-unrelated set $\mathcal T\subseteq \mathcal P_k$ of size at least ${|\mathcal S|\over 14}$ such that for all $\mu\in \mathcal T$, we have $|\supp(\mu)|\leq \lceil C\eps(\eps+\frac 1 n) d)\rceil$.
\end{lemma}

\medskip

Before proving the lemma, we use it to finish the proof of  Theorem~\ref{thm:main}.

\begin{proof}[Proof of Theorem~\ref{thm:main}]
Suppose that there is a map from the $n$-star to $\ell_1^d$ with distortion $1+\eps$.
Then by Lemma~\ref{lem:reduction}, there exists an $\eps$-unrelated set of probability measures on $\{2d+1\}$ of size $n-1$. Thus by Lemma~\ref{lem:one}, there must exist a $\frac 1 2$-unrelated set $\mathcal S$ of probability measures on $\{1,\ldots, 2d+1\}$ of size $\Omega(n)$ such that every measure in $\mathcal S$ has support size at most $$\left\lceil C\cdot{\eps}\cdot\left(\eps+\frac 1 {n-1}\right)\cdot (2d+1)\right\rceil,$$
for some universal constant $C \geq 1$.

We now divide the problem into two cases.
In the case that $C{\eps}(\eps+\frac 1 {|\mathcal S|}) (2d+1)<1$, every measure in $\mathcal S$ is supported on exactly one element, therefore $|\mathcal S| \leq 2d+1$. Hence, 
$$d \geq \Omega(|\mathcal S|) \geq \Omega(n) \geq \frac{\Omega(\log n)}{\e^2 \log(1/\e)}\,,$$
where we have used the assumption that $n \geq 1/\e^2$.

\medskip

In the second case, we have $C{\eps}(\eps+\frac 1 {|\mathcal S|}) (2d+1)\geq1$. Since ${1\over |\mathcal S|}=O(\eps)$, each element $\mu \in \mathcal S$
 has $|\supp(\mu)| \leq O(\e^2 d)$.
 Thus for some constant $c > 0$,
there are at most ${2d+1 \choose c \e^2 d} \leq \exp\left(O(\eps^2d \log(1/\eps)d)\right)$ different supports of size $O(\eps^2 d)$ for the measures in $\mathcal S$.

Since $\mathcal S$ is a $\frac12$-unrelated set of probability measures, for any $\mu,\nu \in \mathcal S$, we have
$$
\|\mu-\nu\|_{TV} \geq \frac12\,.
$$
In particular, if we fix a set $Q \subseteq X$, then by a simple $|Q|$-dimensional volume argument,
$$
|\mu \in \mathcal S : \supp(\mu) \subseteq Q| \leq 3^{|Q|}\,.
$$
All together, we have
$$
|\mathcal S| \leq 3^{O(\e^2 d)} \cdot e^{O(\e^2 d \log(1/\e))} \leq e^{O(\e^2 d \log(1/\e))}\,.
$$
Hence,
$
d \geq \Omega\left({\log |\mathcal S|\over \eps^2\log(1/\eps)}\right)$, completing the proof.
\end{proof}

\begin{remark}
We note that there is a straightforward volume lower
bound for large distortions $D \geq 1$: Any $D$-embedding of the $n$-star
into $\ell_1^d$ requires $d \geq \Omega(\frac{\log n}{\log D})$.
This is simply because the maximal number of disjoint $\ell_1$ balls of radius $1/D$
that can be packed in an $\ell_1$ ball of radius $2$ is $(2D)^d$ in $d$ dimensions.
\end{remark}

We are left to prove Lemma~\ref{lem:one}.
We start by recalling some simple properties of the total variation distance.
For a finite set $S$ and measures $\mu,\nu:2^S\to \mathbb  [0,\infty)$, we define
$$
\min(\mu,\nu) (T)=\sum_{x\in T} \min\left\{\vphantom{\bigoplus}\mu(\{x\}),\nu(\{x\})\right\}.
$$
For $k\in \mathbb N$, and measures  $\mu, \nu\in {\mathcal M}_k$, we have
\begin{equation}\label{eq:min}
\|\mu-\nu\|_{TV}= \frac 1 2(\mu([k])+\nu([k]))-\min(\mu,\nu)([k])\,,
\end{equation}
where we use the notation $[k] = \{1,2,\ldots,k\}$.
We also use the following partial order on measures on the set $S$: $\mu \preceq \nu$, if and only if  for all $T\subseteq S$,  $\mu(T)\leq \nu(T)$.
The following observation is immediate from \eqref{eq:min}.

\begin{observation}\label{obs:one}
Fix $k\in \mathbb N$, $\e > 0$, and measures $\mu,\nu,  \mu', \nu'\in {\mathcal M}_k$, such that $  \mu'\preceq \mu$ and $  \nu'\preceq \nu$.  If
$$\|\mu-\nu\|_{TV}\geq \frac 1 2(\mu([k])+\nu([k]))-\eps,$$ then
\begin{align*}
\|\mu'-\nu'\|_{TV}\geq \frac 1 2(\mu'([k])+ \nu'([k]))-\eps.
\end{align*}
\end{observation}

We will require the following fact in the proof of Lemma~\ref{lem:one}.

\begin{lemma}\label{lem:sparse}
Consider $\delta\in(0,1)$ and a finite subset $S\subseteq [0,\infty)$ such that
\begin{equation}\label{eq:lem:sparse}
\delta\cdot(|S|-1)\cdot \sum_{x\in S}x \geq \sum_{x,y\in S,x\neq y}\min(x,y).
\end{equation}
Then there exists a set $T\subseteq S$, such that $\sum_{x\in T}x\geq \frac 1 2\sum_{x\in S}x$ and $|T|\leq \ceiling{\delta (|S|-1)}$.
\end{lemma}

\begin{proof}
Let $n=|S|$, and let $a_1\geq \cdots \geq a_n\geq 0$ be the elements of $S$ in decreasing order. Then,
$$
\sum_{i=1}^n\sum_{\substack{j=1\\i\neq j}}^n \min(a_i,a_j)=\sum_{i=1}^n\sum_{\substack{j=1\\i\neq j}}^n a_{\max(i,j)}=\sum_{i=1}^n 2(i-1)a_{i}\,.
$$
Letting $k=\ceiling{\delta (|S|-1)}$, we have
$$
\sum_{i=1}^n\sum_{\substack{j=1\\i\neq j}}^n \min(a_i,a_j) \ge \sum_{i=k+1}^n 2(i-1)a_{i} \ge 2k \sum_{i=k+1}^n a_{i}\geq 2\delta(|S|-1)\sum_{i=k+1}^n a_{i}\,.
$$
Combining this inequality and \eqref{eq:lem:sparse} implies that  $\sum_{i=k+1}^n a_{i}\leq \frac 1 2  \sum_{x\in S}x$, therefore $\sum_{i=1}^k a_{i}\geq \frac 1 2  \sum_{x\in S}x.$ Hence the set $T=\{a_1,\ldots, a_k\}$ satisfies both conditions of the lemma.
\end{proof}


\begin{proof}[Proof of Lemma~\ref{lem:one}]
We will show that each of the following statements implies the next one.

\begin{enumerate}
\item[I)] There exists an $\eps$-unrelated set $\mathcal  S\subseteq \mathcal P_k$ of size $n$.

\medskip

\item[II)] There exists an $\eps$-unrelated set $\mathcal S\subseteq {\mathcal M}_k$ of size $n$ such that
\begin{enumerate}
\item for all $\mu\in \mathcal S$, $\mu([k])\leq 1  $;
\item $\sum_{\mu\in \mathcal S}\mu([k])\geq n/4$;
\item $\sum_{\mu\in \mathcal S}|\supp(\mu)|< (2\eps n+1)k$;
\end{enumerate}

\medskip

\item[III)]There exists an $\eps$-unrelated set $\mathcal  S\subseteq {\mathcal M}_k $ of size at least $n/14$ such that
\begin{enumerate}
\item for all $\mu\in \mathcal S$, $|\supp(\mu)|< 14\left(2\eps +{1\over n}\right)k$;
\item for all $\mu\in \mathcal S,$ we have $\mu([k]) \geq 1/8$;
\end{enumerate}

\medskip

\item[IV)] There exists a set satisfying all the conditions of the lemma.
\end{enumerate}

\medskip
\noindent

For ease of notation, given a subset $\mathcal  S\subseteq {\mathcal M}_k$, we define,
$$
\Delta_{\mathcal  S}=\sum_{\mu,\nu\in \mathcal  S,\mu\neq \nu}\min(\mu,\nu).
$$

Note that, if for some $\eps\in[0,1]$, $\mathcal  S \subseteq \mathcal P_k$ is $\eps$-unrelated, then \eqref{eq:min} implies that
\begin{eqnarray}
\Delta_{\mathcal  S}([k])&\leq& \sum_{\mu,\nu\in \mathcal  S,\mu\neq \nu}\frac 1 2 (\mu([k])+\nu([k]))- \|\mu-\nu\|_{TV} \nonumber \\
&\leq& \sum_{\mu,\nu\in \mathcal  S,\mu\neq \nu}\left(1-(1-\eps)\right)\nonumber \\ &=&\eps |\mathcal  S|\cdot(|\mathcal  S|-1)\,.
\label{eq:delta}
\end{eqnarray}


\medskip
\noindent
{\bf I $\Rightarrow$ II:} Suppose that $\mathcal  S_\I\subseteq \mathcal P_k$ is $\eps$-unrelated, and let $X$ be a random variable with state space $\{1,\ldots, k\}$ such that
$$\pr(X=i) = \frac{\sum_{\mu \in \mathcal S_{\I}} \mu(\{i\})}{|\mathcal S_{\I}|}.$$
We have
$$
\E\left[\Delta_{\mathcal S_\I}(X)\over \sum_{\mu\in \mathcal  S_\I}\mu(X)\right]=\frac 1{|\mathcal  S_\I|}\sum_{i=1}^k \Delta_{\mathcal  S_\I}(\{i\})=\frac 1{|\mathcal  S_\I|} \Delta_{\mathcal S_\I}([k])\overset{\eqref{eq:delta}}\leq  \eps(|\mathcal S_\I|-1),
$$
Markov's inequality implies that
\[
\pr\left({\Delta_{\mathcal S_\I}(X)\over \sum_{\mu\in \mathcal  S_I} \mu(X)} \leq 2\eps(|\mathcal  S_\I|-1) \right)\geq \frac 1 2.
\]
So if we let  $$A=\left\{i: 
{\Delta_{\mathcal S_\I}(\{i\})\over \sum_{\mu\in \mathcal S_I}\mu(\{i\})} \leq 2\eps(|\mathcal S_\I|-1)\right\}\,,$$
then we have
\begin{equation}\label{eq:sparse:A}
\frac{1}{|\mathcal S_{\I}|} \sum_{\mu\in \mathcal S_\I} \mu(A)=
\sum_{\mu\in \mathcal S_\I}\sum_{i\in A}
\pr(X=i)
\geq \frac 1 2\,.
\end{equation}

By Lemma~\ref{lem:sparse}, for all $i\in A$ there exists a set $W_i\subseteq \mathcal S_\I$ such that $|W_i|\leq \ceiling{2\eps(|\mathcal  S_\I|-1)}$, and
\begin{equation}\label{eq:sparse:Wi}
\sum_{\mu\in W_i}\mu(\{i\})\geq \frac 1 2\sum_{\mu\in \mathcal S_\I}\mu(\{i\}).
\end{equation}

For $\mu\in \mathcal S_\I$, let $Y_\mu=\{i:\mu\in W_i\}$.  For any $Y \subseteq S$, define $R_Y : 2^S \to 2^S$ by $R_Y(T)=T \cap Y$.
on elements of $Y$ and zero elsewhere.

Let $\mathcal S_{\II}=\{\mu\circ R_{Y_\mu}\}_{\mu\in \mathcal S_\I}$.
Since $\mu\circ R_{Y_\mu}\preceq \mu$,  $\mathcal S_{\II}$ satisfies II(a). Furthermore, Observation~\ref{obs:one} implies that the collection $\{\mu\circ R_{Y_\mu}\}_{\mu\in \mathcal S_\I}$ is $\eps$-unrelated.  Furthermore, $\mathcal S_\II$ satisfies II(b) because
$$\sum_{\mu\in \mathcal S_I}{\mu(Y_\mu)}\overset{\eqref{eq:sparse:Wi}}{\geq} \frac 1 2 \sum_{i\in A}\sum_{\mu\in \mathcal S_I}\mu(\{i\})=\frac 1 2\sum_{\mu\in \mathcal S_I}\mu(A)\overset{\eqref{eq:sparse:A}}\geq \frac14 |\mathcal S_\I|\,.$$
Finally, condition II(c) holds because
$$
 \sum_{\mu\in \mathcal S_\I} |\supp(\mu\circ  R_{Y_\mu})| \leq \sum_{i\in A} |W_i|< 2\eps(|\mathcal S_\I|-1) |A|+|A|\leq (2\eps |\mathcal S_\I|+1) k\,.
$$

\medskip
\noindent
{\bf II $\Rightarrow$ III:}
Suppose that $\mathcal S_{\II} \subseteq \mathcal M_k$ is an $\eps$-unrelated collection of cardinality $n$ satisfying all the conditions of II.
We have $\max\{\mu([k])\}_{\mu\in \mathcal S_{\II}}\leq 1$ and $\sum_{\mu\in \mathcal  S_{\II}} \mu([k])\geq |\mathcal S_{\II}|/4$. Therefore, there exists a subcollection $\mathcal S'\subseteq \mathcal S_{\II}$ such that for all $\mu\in  \mathcal S'$, we have $\mu([k])\geq 1/8$, and
$$
|\mathcal S'|\geq \left({1/4-1/8\over1-1/8}\right)|\mathcal S_{\II}| \geq \frac n 7.
$$

By Markov's inequality, there exists a collection of measures $\mathcal S_{\III}$ such that $|\mathcal S_{\III}|\geq \frac12|\mathcal S'|\geq \frac 1{14}|\mathcal S_{\II}|$, where  for all $\mu\in \mathcal S_{\III}$,
\begin{eqnarray*}
\supp(\mu)\leq
 2{\sum_{\mu\in \mathcal S'}|\supp(\mu)|\over |\mathcal S'|}&\leq&  2{\sum_{\mu\in \mathcal S_{\II}}|\supp(\mu)|\over |\mathcal S'|} \\
 &\leq&  2{\sum_{\mu\in S_{\II}}|\supp(\mu)|\over |S_{\II}|/ 7}\overset{\textrm{II(c)}}\leq
14k\left(2\eps+{1\over n}\right).
\end{eqnarray*}

The set $\mathcal S_{\III}$ has size at least $n\over 14$ and by construction satisfies conditions (a) and (b) of III.

\medskip
\noindent
{\bf III $\Rightarrow$ IV:}
Suppose $\mathcal S_\III \subseteq \mathcal M_k$ is a an $\eps$-unrelated collection of cardinality at least $n/14$.
For each measure $\mu\in \mathcal S_{\III}$, let $Z_\mu\subseteq \{1,\ldots, k\}$ be the set of  $\ceiling{16\cdot\eps\left(14k({2\eps+{1\over n})}\right)}$ elements of $\{1,\ldots, k\}$ that has the largest measures with respect to $\mu$ (breaking ties arbitrarily).
Since ${\eps\le \frac 1 8}$, for all $\mu\in \mathcal S_\III$ we have
\begin{equation}\label{eq:l3:eps}
\mu_{}(Z_\mu)\ge \frac 1 8\left(16\cdot \eps(14k({2\eps+{1\over n})})\over 14k({2\eps+{1\over n})}\right)= 2\eps.
\end{equation}
Let $\mathcal S_{\IV}=\left\{{\mu\circ R_{Z_\mu}\over \mu({Z_\mu})} : \mu \in S_{\III}\right\}$.  Clearly $\mathcal S_{\IV}\subseteq \mathcal P_k$, and  $|\mathcal S_{\IV}| \geq {n\over 14}$. Moreover, by our construction for all $\bar \mu\in \mathcal S_{\IV}$,
 $|\supp(\bar \mu)|\leq \lceil 224\eps(2\eps+\frac 1 n) k)\rceil$. To complete the proof we need to show $\mathcal S_{\IV}$ is $\frac 1 2$-unrelated.
 Note that if $\mu,\nu\in \mathcal S_{\III}$, then Observation~\ref{obs:one} implies that
$$\min(\nu\circ R_{Z_{\nu}},\mu\circ R_{Z_{\mu}})([k])\overset{\eqref{eq:min}}={{{\mu({Z_{\mu}})+\nu({Z_{\mu}})}\over 2}-\|\mu \circ R_{ Z_{\nu}}-\mu \circ R_{ Z_{\nu}}\|_{TV}}\leq \eps.$$
Therefore,
\begin{align*}
\left\|{\mu \circ R_{ Z_\mu}\over \mu({Z_\mu})}- {\nu \circ R_{ Z_\nu}\over \nu({Z_\nu})}\right\|_{TV} &
\overset{\eqref{eq:min}}= 1- \min\left({\mu \circ R_{ Z_\mu}\over \mu({Z_\mu})}, {\nu \circ R_{ Z_\nu}\over \nu({Z_\nu})}\right)([k])\\
&\overset{\eqref{eq:l3:eps}}\geq 1-{{\min\left({\mu \circ R_{ Z_\mu}\over 2\eps}, {\nu \circ R_{ Z_\nu}\over 2\eps}\right)}}([k])
\overset{\eqref{eq:min}} \geq \frac 1 2\,,
\end{align*}
completing the proof.
\end{proof}

\section{Nearly equilateral sets in $\ell_1^d$}
\label{sec:kahane}

We will need the following result of Kahane \cite{Kahane81}.

\begin{theorem}\label{thm:kahane81}
For every $\e \in (0,1)$, there exists a mapping $K_{\e} : \mathbb R \to \ell_2^d$ such that $d \leq O(1/\e)$
and the following holds:  For every $x,y \in \mathbb R$,
$$
(1-\e) \sqrt{|x-y|} \leq \|K_{\e}(x)-K_{\e}(y)\|_2 \leq \sqrt{|x-y|}\,.
$$
\end{theorem}

\begin{proof}[Proof of Theorem \ref{thm:kahane}]
Suppose that $f : U_n \to \ell_1^d$ is a $(1+\e)$-embedding scaled so that $f$ is $1$-Lipschitz.
Consider
the mapping $g : U_n \to \ell_2^{O(d/\e)}$ given by
$$
g(x) = \left(\vphantom{\bigoplus} K_{\e}(f(x)_i), K_{\e}(f_2(x)_i), \ldots, K_{\e}(f_d(x)_i)\right)\,,
$$
where $f(x)_i$ denotes the $i$th coordinate of $f(x)$.  By Theorem \ref{thm:kahane81}, for any $x,y \in U_n$, we have
$\|g(x)-g(y)\|^2_2 \leq \|f(x)-f(y)\|_1$.  On the other hand,
$$
\|g(x)-g(y)\|_2^2 \geq (1-\e)^2 \|f(x)-f(y)\|_1 \geq \frac{(1-\e)^2}{1+\e}\,,
$$
implying that $\|g(x)-g(y)\|_2 \geq (1-\e)/\sqrt{1+\e} \geq 1-2\e$.  Thus $g$ is a $(1+2\e)$-embedding
of $U_n$ into $\ell_2^{O(d/\e)}$.  But now by \cite{Alon03}, for $n$ sufficiently large, we have
$$
d \geq \frac{\Omega(\log n)}{\e \log \frac{1}{\e}}\,.
$$
\end{proof}

\section{Extension to $k$-ary trees}
\label{sec:kary}

We now prove Corollary \ref{cor:kary}.  Combining the next lemma with Theorem \ref{thm:main} yields the desired result.

\begin{lemma}\label{lem:reduction:2}
For $h, k\geq 2$, let $B_{h,k}$ be a complete $k$-ary tree of height $h$.
If $B_{h,k}$ admits a $(1+\e)$-embedding into $\ell_1^d$ for some $0 \leq \e \leq \frac18$, then
the $(1+k^{\lceil h/2\rceil})$-star admits a $(1+4\e)$-embedding into $\ell_1^d$.
\end{lemma}

\begin{proof}
Suppose that $f : B_{h,k} \to \ell_1^d$ is a $(1+\e)$-embedding.
We may assume, without loss, that $f$ is $1$-Lipschitz.
Letting $n = (1+k^{\lceil h/2\rceil})$, we construct
an embedding $g : V_n \to \ell_1^d$ of the $n$-star as follows.

Let $r \in V_n$ denote the vertex of degree $n-1$.  We put $g(r)=0$.
Let $S$ be the set of vertices in $B_{h,k}$ at height $\lceil h/2\rceil$ (we use the convention that root has height zero).
For any vertex $v\in S$, pick an arbitrary leaf $x_v$ in the subtree rooted at $v$.
Associate to every vertex $w \in V_n \setminus \{r\}$ a distinct element $\tilde w \in S$ and put
$$
g(w) = \frac{f(x_{\tilde w}) - f(\tilde w)}{h- \lceil h/2\rceil}\,.
$$

Since $f$ is $1$-Lipschitz, the same holds for $g$.
Moreover for any two distinct elements $u,v\in S$ we have
\begin{align*}
2(h-\lceil h/2\rceil)+d_{B_{h,k}}(u,v)
&=d_{B_{h,k}}(u,v)+d_{B_{h,k}}(x_v,v)+d_{B_{h,k}}(x_u,u)\\
&=d_{B_{h,k}}(x_u,x_v)\\
&\leq (1+\eps)\|f(x_u)-f(x_v)\|_1\\
&\leq (1+\eps)\|(f(x_u)-f(u))-(f(x_v)-f(v))\|_1 \\ & \qquad + (1+\eps)\|f(u)-f(v)\|_1\\
&\leq (1+\eps)\|(f(x_u)-f(u))-(f(x_v)-f(v))\|_1\\ & \qquad + (1+\eps)\,d_{B_{h,k}}(u,v)\,.
\end{align*}
Therefore,
\begin{align*}
(1+\eps)\|(f(x_u)-f(u))-(f(x_v)-f(v))\|_1&\geq 2(h-\lceil h/2\rceil)- \eps d_{B_{h,k}}(u,v)\\
&\geq 2(h-\lceil h/2\rceil)- 2\eps \lceil h/2\rceil\\
&\geq 2(h-\lceil h/2\rceil)- 4\eps (h-\lceil h/2\rceil)\\
 &\geq (2-4\eps)(h-\lceil h/2\rceil).
\end{align*}
Since $\eps\leq 1/8$, the preceding inequality bounds the distortion of $g$ by
$
\frac {1+\eps} {1-2\eps}\leq 1+4\eps\,,
$
completing the proof.
\end{proof}

\subsection*{Acknowledgments}

We are grateful to Jiri Matou{\v{s}}ek for suggesting
the approach of Section \ref{sec:kahane}. 
This research was partially supported by NSF grant CCF-1217256.

\bibliographystyle{plain}
\bibliography{socg}

\end{document}